\documentclass[12pt]{article}

\usepackage{amsmath}
\usepackage{amssymb,amsthm}
\usepackage{microtype}
\usepackage{amsthm,amsfonts}
\usepackage{mathrsfs,microtype}
\usepackage[margin=1.3in]{geometry}
\usepackage[usenames]{color}
\usepackage{graphicx}
\usepackage{tikz}
\usepackage{verbatim}
\usepackage{mathdots}
\usepackage{mathtools}
\usetikzlibrary{arrows,decorations.pathmorphing,decorations.footprints,fadings,calc,trees,mindmap,shadows,decorations.text,patterns,positioning,shapes,matrix,fit,through,decorations.pathreplacing}
\usepackage{hyperref}

\theoremstyle{plain}
\newtheorem{theorem}{Theorem}[section]
\newtheorem{lemma}[theorem]{Lemma}
\newtheorem{corollary}[theorem]{Corollary}

\theoremstyle{definition}
\newtheorem{definition}[theorem]{Definition}

\newtheorem*{notation}{Notation}

\theoremstyle{remark}
\newtheorem*{remark}{Remark}

\newcommand{\frakM}{\mathfrak{M}} %AH Draft 1 1-20-14

\newcommand{\calC}{\mathcal{C}}
\newcommand{\calI}{\mathcal{I}}
\newcommand{\calD}{\mathcal{D}}

\newcommand{\frakS}{\mathfrak{S}} %AH Draft 1 1-20-14

\DeclarePairedDelimiter{\ceil}{\lceil}{\rceil}

\title{Trivial %AH Draft 1 1-20-14
Meet and Join within the Lattice of Monotone Triangles}
\author{John Engbers\thanks{john.engbers@marquette.edu; Department of Mathematics, Statistics and Computer Science, Marquette University, Milwaukee, WI 53201}
\and
Adam Hammett\thanks{adam.hammett@bethelcollege.edu; Department of Mathematical Sciences, Bethel College, Mishawaka, IN 46545}} 
\date{\today }

\begin{document}

\maketitle

\begin{abstract}
\noindent The lattice of monotone triangles $(\frakM_n,\le)$ ordered by entry-wise comparisons is studied. Let $\tau_{\min}$ denote the unique minimal element in this lattice, and $\tau_{\max}$ the unique maximum. The number of $r$-tuples of monotone triangles $(\tau_1,\ldots,\tau_r)$ with minimal infimum $\tau_{\min}$ (maximal supremum $\tau_{\max}$, resp.) is shown to asymptotically approach $r|\frakM_n|^{r-1}$ as $n \to \infty$. Thus, with high probability this event implies that one of the $\tau_i$ is $\tau_{\min}$ ($\tau_{\max}$, resp.). Higher-order error terms are also discussed.
\end{abstract}

\section{Introduction and statement of results}

Let $n \geq 1$ be an integer and $[n]:=\{1,2,\ldots,n\}$.  A \emph{monotone triangle} of size $n$ (or %more specifically a 
 \emph{Gog triangle} in the terminology of Doron Zeilberger \cite{Zeilberger}) is a triangular arrangement of $n(n+1)/2$ integers with $i$ elements in row $i$ ($i\in [n]$) taken from the set $[n]$ so that if $a(i,j)$ denotes the $j\text{th}$ entry in row $i$ (counted from the top), then
\begin{align*}
    &1.~ a(i,j) < a(i,j+1), \,\, 1 \leq j < i \textrm{, and}\\
    &2.~ a(i,j) \leq a(i-1,j) \leq a(i,j+1), \,\, 1 \leq j < i.
\end{align*}
Note that this definition forces the last row to be $a(n,1)=1$, $a(n,2)=2$, $\dots$, $a(n,n)=n$. We let $\frakM_n$ denote the set of all monotone triangles of size $n$, and we let $\tau = (\tau(i,j))$ be a generic element from this set.  It is well-known that there is a bijection between $\frakM_n$ and the set of all $n\times n$ alternating-sign matrices (ASMs), which are the $n \times n$ matrices of $0$s, $1$s, and $-1$s so that the sum of the entries in each row and column is $1$ and the non-zero entries in each row and column alternate in sign.  This bijection is often seen via the column-sum matrix form of an ASM; an example of a monotone triangle of size $4$, along with the corresponding column-sum matrix and ASM, is given in Figure \ref{fig-MT}.

\begin{figure}[ht]
\begin{center}
\begin{tabular}{ ccccccccc }
   &           &       &           &     3     &          &       &              \\
   &           &       &     2     &           &     4    &       &              \\
   &           &   1   &           &     3     &          &   4   &              \\
%   & $\iddots$ &       & $\iddots$ &           & $\ddots$ &       & $\ddots$ &    \\
  &     1  &      &     2      &   &        3  &  &     4
\end{tabular}
$\iff$
$\left(
\begin{tabular}{cccc}
0 & 0 & 1 & 0\\
0 & 1 & 0 & 1\\
1 & 0 & 1 & 1\\
1 & 1 & 1 & 1\\
\end{tabular}\right)$
$\iff$
$\left(
\begin{tabular}{cccc}
0 & 0 & 1 & 0\\
0 & 1 & -1 & 1\\
1 & -1 & 1 & 0\\
0 & 1 & 0 & 0\\
\end{tabular}\right)$
\end{center}
\caption{A monotone triangle of size $4$, along with the corresponding column-sum and alternating-sign matrices.}
\label{fig-MT}
\end{figure}
The set $\frakM_n$ is also in bijection with the square ice configurations (or six-vertex configurations) of order $n$ in statistical mechanics; for descriptions of these bijections, along with other sets in bijection with $\frakM_n$, see the surveys of James Propp \cite{ProppFaces}, David Bressoud and Propp \cite{BressoudPropp}, or the wonderful exposition by Bressoud about the proof of the ASM conjecture \cite{Bressoud}, which concerns the \emph{number} of monotone triangles, denoted $A(n):=|\frakM_n|$ (as this was first framed for ASMs). Specifically, the ASM conjecture is simply that
\begin{equation}\label{eqn-A(n)formula}
A(n) = \prod_{k=0}^{n-1} \frac{(3k+1)!}{(n+k)!}.
\end{equation}
The formula (\ref{eqn-A(n)formula}) was first proved by Zeilberger \cite{Zeilberger}, with another proof (using square ice) given by Greg Kuperberg \cite{Kuperberg}. Bressoud's book \cite{Bressoud} not only lays out this proof in complete detail, but also delivers the historical narrative surrounding this triumph.

%While the formula (\ref{eqn-A(n)formula}) is lovely to behold, it is not practically useful in the realm of asymptotics. For instance, while it is certainly true that the number of ASMs supersedes the number of permutations of order $n$ (as $n\times n$ permutation matrices are, after all, ASMs), is it even clear directly from (\ref{eqn-A(n)formula}) that $A(n)$ is \emph{much} larger than $n!$, the total number of order-$n$ permutations? Fortunately, there do exist useful asymptotic formulas for $A(n)$ in the literature.  Indeed, 
For the asymptotic value of $A(n)$, Bleher and Fokin \cite{BleherFokin} found that
\begin{equation*}%\label{A(n)asymptotics}
A(n) = c\left( \frac{3\sqrt{3}}{4} \right)^{n^2}\times n^{-5/36}\times \left( 1 - \frac{115}{15552n^2} + O\left( n^{-3} \right)  \right),
\end{equation*}
where $c>0$ is a constant.  Using Stirling's approximation $n! \sim e^{n\log n - n}\sqrt{2\pi n}$ it becomes clear that $A(n)$ and $n!$ are far apart for large $n$. We shall not need this in our present study, however.  We simply need the fact that $A(n) \geq n!$, which we have already seen is clear from permutation matrices being a subset of the ASMs, and can also be seen directly from monotone triangles  (by iteratively building a monotone triangle from the top to the bottom, and for each new row adding one new number to the previous row). Before departing from the question of enumerating objects related to monotone triangles, we also mention some recent work of Ilse Fischer \cite{FischerMTBottomRow,FischerMTBottomRowSimp,FischerMTTrapezoid} which has focused on finer enumerative questions regarding a looser definition of monotone triangles where the bottom row is simply required to be a strictly increasing collection of $n$ positive integers (rather than the \emph{specific} integers $1, 2, \dots, n$), and also the enumeration of so-called \emph{monotone trapezoids}. %AH Draft 7-8-14 start
Also, Jessica Striker \cite{Striker} has recently settled the long-sought question of a bijection between descending plane partitions and ASMs in the case where the ASM is simply a permutation matrix. Striker's bijection involves a particular weighting of permutation inversions, a concept that we explore further in a forthcoming work extending this present study.
%AH Draft 7-8-14 end

Our present work will focus exclusively on the lattice properties of $(\frakM_n,\le)$, where ``$\le$'' in this poset is defined by entry-wise comparisons. That is, given monotone triangles $\tau_1 = (\tau_1(i,j))$ and $\tau_2 = (\tau_2(i,j))$ in $\frakM_n$, we define $\tau_1 \le \tau_2$ if and only if $\tau_1(i,j) \le \tau_2(i,j)$ for $1\le j \le i \le n$. Of course, it is entirely possible that two monotone triangles may not be comparable at all. For example, in the case $n=3$ the following two monotone triangles are incomparable:
\begin{equation}\label{tau1-tau2-ex}
\tau_1 :=
\begin{tabular}{ ccccc }
          &           &     3     &          &                   \\
          &     1     &           &     3    &                   \\
      1   &           &     2     &          &   3               \\
\end{tabular}
\quad \textrm{ and } \quad \tau_2 := 
\begin{tabular}{ ccccc }
          &           &     2     &          &                   \\
          &     2     &           &     3    &                   \\
      1   &           &     2     &          &   3               \\
\end{tabular}.
\end{equation}
Indeed, all entries in both monotone triangles are the same except for the first entries in rows $1$ and $2$, and there we have $\tau_1(1,1)>\tau_2(1,1)$ and yet $\tau_1(2,1)<\tau_2(2,1)$. However, %JE Draft 1 1-21-2014 changed from `moreover', changed wording later in the sentence
this poset is indeed a lattice under the infimum and supremum operations
\[
\inf(\tau_1,\tau_2):=(\min\{\tau_1(i,j),\tau_2(i,j)\}) \textrm{\,\, and \,\,}
\sup(\tau_1,\tau_2):=(\max\{\tau_1(i,j),\tau_2(i,j)\}).
\]
Infimums and supremums for $r$ monotone triangles, where $r > 2$, are defined analogously. It is left to the interested reader to check that these inf and sup operations do deliver monotone triangles, and satisfy the requirements of an infimum or supremum. 

Interestingly, if we restrict this ordering to the collection of order-$n$ permutations $\frakS_n \subseteq \frakM_n$ %AH Draft 1 1-20-14
with each permutation identified with its corresponding monotone triangle (through the correspondence described above), %$n\times n$ permutation matrix (and hence monotone triangle via the correspondence desribed above), 
we obtain a poset $(\frakS_n,\le)$ %AH Draft 1 1-20-14, %JE Draft 1 1-21-2014 I changed this to identify the permutation with the monotone triangle instead of the permutation matrix
referred to as the \emph{Bruhat ordering} on $\frakS_n$ %AH Draft 1 1-20-14
(see, e.g., H. and Boris Pittel's work \cite{HammettPittel} or \cite{Hammett} for an exposition on the comparability properties of $(\frakS_n,\le)$). %AH Draft 1 1-20-14
It was C. Ehresmann \cite{Ehresmann} who first discovered this connection between $\frakS_n$ %AH Draft 1 1-20-14
under Bruhat ordering and their monotone triangle counterparts under entry-wise comparisons. It is worth mentioning here that the infimum of two elements of $\frakS_n$, when viewed as elements of $\frakM_n$, need not lie in $\frakS_n$; %However, the lattice property enjoyed by $(\frakM_n,\le)$ is not inherited by $(\frakS_n,\le)$. %AH Draft 1 1-20-14
indeed, consider the $n=3$ example above in  (\ref{tau1-tau2-ex}). There we have
\[
\inf(\tau_1,\tau_2)= 
\begin{tabular}{ ccccc }
          &           &     2     &          &                   \\
          &     1     &           &     3    &                   \\
      1   &           &     2     &          &   3               \\
\end{tabular},
\]
and it is easy to see that this is not a monotone triangle arising from any permutation in $\frakS_3$ even though both $\tau_1$ and $\tau_2$ do (they correspond to the permutations $312$ and $231$, respectively; a monotone triangle coming from a permutation will necessarily have the property that each row is a subset of the very next row). %AH Draft 7-8-14 start
Moreover, when the permutations $312$ and $231$ are viewed as elements of the Bruhat poset $(\frakS_3,\le)$ there are two potential infimums, namely $132$ and $213$. However, these two permutations are incomparable in this poset and thus $(\frakS_3,\le)$ is not a lattice. This is a general feature of the Bruhat permutation-poset embedded within the monotone triangle lattice: the lattice property enjoyed by $(\frakM_n,\le)$ is not inherited by $(\frakS_n,\le)$. %AH Draft 7-8-14 end
In fact, Richard Stanley notes in exercise 7.103 of \cite{Stanley} that $(\frakM_n,\le)$ is the unique \emph{MacNeille completion} of $(\frakS_n,\le)$ to a lattice. A. Lascoux and M.~P. Sch\"utzenberger \cite{LascouxSchutz} were the first to prove this fact. We now focus exclusively on a particular feature of this lattice.

Let $\tau_{\min}$ denote the (unique) minimal element in the monotone triangle lattice. That is,
\[
\tau_{\min} \coloneqq
\begin{tabular}{ ccccccccc }
   &           &       &           &     1     &          &       &          &    \\
   &           &       &     1     &           &     2    &       &          &    \\
   &           &   1   &           &     2     &          &   3   &          &    \\
   & $\iddots$ &       & $\iddots$ &           & $\ddots$ &       & $\ddots$ &    \\
 1 &           &   2   &           & $\cdots$  &          & $n-1$ &          & $n$
\end{tabular}.
\]
Similarly, let $\tau_{\max}$ denote the (unique) maximal element:
\[
\tau_{\max} \coloneqq
\begin{tabular}{ ccccccccc }
   &           &       &           &     $n$     &          &       &          &    \\
   &           &       &     $n-1$     &           &     $n$    &       &          &    \\
   &           &   $n-2$   &           &     $n-1$     &          &   $n$   &          &    \\
   & $\iddots$ &       & $\iddots$ &           & $\ddots$ &       & $\ddots$ &    \\
 1 &           &   2   &           & $\cdots$  &          & $n-1$ &          & $n$
\end{tabular}.
\]
In this paper, we are interested in two questions. How likely is it that $r$ independent and uniformly random monotone triangles $\tau_1,\ldots,\tau_r\in \frakM_n$ will have \emph{trivial meet}, i.e.~$\inf(\tau_1,\ldots,\tau_r) = \tau_{\min}$, and how likely are they to have \emph{trivial join}, i.e.~$\sup(\tau_1,\ldots,\tau_r) = \tau_{\max}$? Namely, we want \emph{sharp} estimates on the probabilities
\begin{align*}
&1.~ p_{\min} \coloneqq P(\inf (\tau_1,\ldots,\tau_r) = \tau_{\min})
\textrm{, and}\\
&2.~ p_{\max} \coloneqq P(\sup (\tau_1,\ldots,\tau_r) = \tau_{\max}).
\end{align*}
These questions have previously been studied for the lattice of set partitions ordered by refinement (see Canfield \cite{Canfield} and Pittel \cite{PittelMeetJoin}), the lattice of set partitions of type $B$ (see Chen and Wang \cite{ChenWang}), and the lattice of permutations under weak ordering (see H. \cite{Hammett}). 

First of all, we show that 
\begin{equation}\label{eqn-minmax}
p_{\min} = p_{\max},
\end{equation}
reducing our problem to the consideration of $p_{\min}$ only.

\begin{proof}[Proof of (\ref{eqn-minmax})]
Given a monotone triangle $\tau=(\tau(i,j))\in\frakM_n$, introduce the \emph{rank-reversed} monotone triangle $\bar{\tau}\in \frakM_n$ obtained by the transformation of entries
\[
\tau(i,j)\mapsto n-\tau(i,j)+1
\]
and then reversing the order of each row. For example ($n=4$)
\[
\begin{array}{c}
\tau =
\begin{tabular}{ ccccccc }
  &   &   & 3 &   &   &   \\
  &   & 2 &   & 3 &   &   \\
  & 1 &   & 2 &   & 4 &   \\
1 &   & 2 &   & 3 &   & 4 \\
\end{tabular}
\mapsto
\bar{\tau} =
\begin{tabular}{ ccccccc }
  &   &   & 2 &   &   &   \\
  &   & 2 &   & 3 &   &   \\
  & 1 &   & 3 &   & 4 &   \\
1 &   & 2 &   & 3 &   & 4 \\
\end{tabular}.
\end{array}
\]
Certainly if $\tau_1,\ldots,\tau_r\in \frakM_n$ are uniformly random and independent, the same is true of $\bar{\tau_1},\ldots,\bar{\tau_r}\in \frakM_n$. Moreover, it is easy to check that
\[
\inf (\tau_1,\ldots,\tau_r) = \tau_{\min} \quad \iff \quad \sup (\bar{\tau_1},\ldots,\bar{\tau_r}) = \tau_{\max}.
\]
But this means that $p_{\min} = p_{\max}$.
\end{proof}

\noindent We are ready to state our main result.

\begin{theorem}\label{thm-main}
Fix an integer $r\ge 1$. Given $r$ independent and uniformly random monotone triangles $\tau_1,\ldots,\tau_r \in \frakM_n$, we have
\[
p_{\min} \sim \frac{r}{A(n)}, \,\, n\to\infty.
\]
\end{theorem}

\noindent Notice that the theorem is obvious when $r=1$, since in this case $p_{\min} = 1/A(n)$.  Because of this, we assume that $r\geq 2$ for the remainder of the paper; this theorem says that essentially nothing changes for these $r$: almost all ways of obtaining $\tau_{\min}$ occur when some $\tau_i = \tau_{\min}$!

It is possible to extend the proof of Theorem \ref{thm-main} to %AH Draft 1 1-20-14
higher-order terms.  In particular, we will describe the modifications needed to obtain the following refinement.

\begin{theorem}\label{thm-main2}
Fix an integer $r\ge 1$. Given $r$ independent and uniformly random monotone triangles $\tau_1,\ldots,\tau_r \in \frakM_n$, we have
\[
p_{\min} = \frac{rA(n)^{r-1} + 2r(r-1)A(n-1)A(n)^{r-2} + \Theta(A(n-2)A(n)^{r-2})}{A(n)^r}.
\]
\end{theorem}

\noindent Theorem \ref{thm-main} is a simple corollary of Theorem \ref{thm-main2} (as $A(n-1) = o(A(n))$ and $A(n-2) = o(A(n-1))$ from Lemma \ref{lem-ratio}), but we leave them in the paper as separate results. This is because the fundamental insights that give rise to Theorem \ref{thm-main} are obfuscated when refined for the proof of Theorem \ref{thm-main2}. Also, the careful reader will notice that the proof of Theorem \ref{thm-main2} could in theory be extended to deliver further higher-order terms. However, the argument becomes %notoriously 
more complex and involves a much finer case-by-case treatment.
We hope that perhaps there is another method of proof that quickly gives all terms of $p_{\min}$.  In particular, other proofs of similar statements (e.g. those in  \cite{Canfield,Hammett,PittelMeetJoin}) proceed via generating functions, and it would be nice to find a different proof along these lines. 
% We hope that there is an alternate method, perhaps via generating functions, more amenable to a higher-order term analysis.

\medskip

The remainder of the paper is laid out as follows.  In Section \ref{sec-proof}, we first prove some introductory lemmas about the values $A(n)$ and then use these lemmas to prove Theorem \ref{thm-main}.  Section \ref{sec-higher order terms} describes how to modify this proof to obtain Theorem \ref{thm-main2}.  %Finally, in Section \ref{sec-questions}, we describe some other related and interesting questions. 

\section{Proof of Theorem \ref{thm-main}}\label{sec-proof}
\subsection{Lemmas}

Before starting the proof of Theorem \ref{thm-main}, we begin with several lemmas that help us understand how the values $A(n)$ relate to one another. Throughout we use the convention that $A(0)=1$.

\begin{lemma}\label{lem-increaseAi}
Let $i_1 \geq i_2 \geq 1$.  Then $A(i_1+1)A(i_2-1) \geq A(i_1)A(i_2)$.
\end{lemma}

\begin{proof}
From (\ref{eqn-A(n)formula}) we have $A(n)/A(n-1)=\frac{(3n-2)!(n-1)!}{(2n-1)!(2n-2)!}$.  As this is an increasing function of $n$ for $n \geq 1$, the lemma follows; we leave the details to the reader.
%Recall from (\ref{eqn-A(n)formula}) that
%\[
%A(n) = \displaystyle\prod_{k=0}^{n-1} \frac{(3k+1)!}{(n+k)!}.
%\]
%To prove the lemma, notice that it suffices to show that
%\begin{equation*}
%\frac{A(n)}{A(n-1)} = \frac{(3n-2)!(n-1)!}{(2n-1)!(2n-2)!}
%\end{equation*}
%is an increasing function of $n$ for $n \geq 1$.
%But $A(n+1)/A(n) \geq A(n)/A(n-1)$ is equivalent to
%\[
%\frac{(3n+1)!(n)!}{(2n+1)!(2n)!} \geq \frac{(3n-2)!(n-1)!}{(2n-1)!(2n-2)!},
%\]
%or
%\[
%\frac{(3n+1)(3n)(3n-1)}{(2n+1)(2n)(2n-1)} \geq 2
%\]
%which reduces to
%\[
%\frac{9n^2-1}{4n^2-1} \geq \frac{4}{3}.
%\]
%It is easy to check that this holds for $n\geq 1$. Indeed, 
%\[
%\frac{9n^2-1}{4n^2-1} = 2+\frac{n^2+1}{4n^2-1}>2+\frac{1}{4}>\frac{4}{3}.
%\]
\end{proof}

\begin{lemma}\label{lem-ratio}
%For all $n \geq 1$,
%\[
%\frac{A(n-1)}{A(n)} \le \left(\frac{2}{3}\right)^{n-1}.
%\]
For all $n \geq 1$ and any constant $c \in [n]$, we have
\begin{equation*}
\frac{A(n-c)}{A(n)} \leq \left( \frac{2}{3} \right)^{\binom{n}{2} - \binom{n-c}{2}} =\left(\frac{2}{3}\right)^{c(2n-c-1)/2}.
\end{equation*}
\end{lemma}

\begin{proof} %AH Draft 1 1-20-14 (changed n to t)
First of all, for $t\ge 1$ we have
\[
\frac{A(t-1)}{A(t)} =
\frac{(2t-1)!(2t-2)!}{(3t-2)!(t-1)!}
    = \frac{(2t-2)(2t-3)\cdots(t)}{(3t-2)(3t-3)\cdots(2t)}
\]
and so
\begin{equation}\label{eqn-int step}
\frac{A(t-1)}{A(t)} \leq \left( \frac{2(t-1)}{3(t-1)+1} \right)^{t-1} 
    \le \left( \frac{2}{3} \right)^{t-1}.
\end{equation}
But then
\[
\frac{A(n-c)}{A(n)} = \frac{A(n-c)}{A(n-c+1)} \cdot \frac{A(n-c+1)}{A(n-c+2)} \cdot \cdots \cdot \frac{A(n-1)}{A(n)},
\]
and the result follows from using (\ref{eqn-int step}) on each term.
\end{proof}

\begin{remark}
The bound given in Lemma \ref{lem-ratio} can be significantly sharpened using either Stirling-type estimates or the asymptotic formula for $A(n)$, but we will only need this simple bound here. %Also, note that this estimate becomes a \emph{strict} inequality for $n>1$.
\end{remark}

\begin{definition}
Call row $i_0$ in a monotone triangle $\tau=(\tau(i,j))\in\frakM_n$ \emph{distinguished} if it consists of the smallest possible entries $1$, $2$, $3$, $\ldots$, $i_0$, i.e.,
\[
\tau(i_0,j) = j \textrm{ for } 1\le j \le i_0.
\]
\end{definition}

\noindent Notice that row $n$ is distinguished for any monotone triangle $\tau \in \frakM_n$.

\begin{notation}
Let $\calD(\tau)$ denote the set of distinguished rows for a given monotone triangle $\tau \in \frakM_n$. Given a collection of row indices $\calI = \{i_1,i_2,\ldots,i_k\} \subseteq [n-1]$, let % with $|\calI|=k$, let %%%\calI_< not used? $\calI_<:=\{i_1<i_2<\cdots<i_k\}$ be the elements of $\calI$ indexed in ascending order. Let 
$\eta_n(\calI)=\eta_n(i_1,i_2,\ldots,i_k)$ denote the number of monotone triangles 
$\tau\in\frakM_n$ such that $\calI \cup\{n\}\subseteq \calD(\tau)$.
When applicable, we will assume that $i_1 < i_2 < \cdots < i_k$.  If no distinguished rows are specified outside of the (bottom) $n$-th row, 
we adopt the notation $\eta_n(\emptyset)$. That is, $\eta_n(\emptyset)$ is the number of 
monotone triangles $\tau\in\frakM_n$ such that 
$\emptyset\cup\{n\}\subseteq \calD(\tau)$; clearly we have $\eta_n(\emptyset)=A(n)$.  
\end{notation}

\begin{lemma}\label{lem-distinguishedrows} For any row indices $i_1,i_2,\ldots,i_k \in [n-1]$, we have
\[
\eta_n(i_1,i_2,\ldots,i_k)=A(i_1)A(i_2-i_1)\cdots A(i_k-i_{k-1})A(n-i_k).
\]
\end{lemma}

\begin{proof}
%If row $i_1$ is distinguished then the top $i_1$ rows form a monotone triangle of size $i_1$, hence the first factor of $A(i_1)$. In fact, we have
We first claim that
\begin{equation}\label{1st-eta-iterate}
\eta_n(i_1,i_2,\ldots,i_k)=A(i_1)\eta_{n-i_1}(i_2-i_1,\ldots,i_k-i_1).
\end{equation}
Indeed, if row $i_1$ is distinguished then the top $i_1$ rows form a monotone triangle of size $i_1$, hence the first factor of $A(i_1)$. The bottom $n-i_1$ rows must all start with the numbers $1$, $2$, $3$, $\ldots$, $i_1$ in their first $i_1$ entries; the remaining entries in these rows are in bijective correspondence with a monotone triangle of size $n-i_1$ given by subtracting $i_1$ from each of these entries, and row $i$ among the last $n-i_1$ rows in the original monotone triangle 
corresponds to row $i-i_1$ in this bijective element. As the rows $i_2$, $\ldots$, $i_k$, and $n$ are distinguished in the original monotone triangle, this means rows with transformed indices $i_2-i_1$, $\ldots$, $i_k-i_1$, and $n-i_1$ are distinguished in the corresponding monotone triangle of size $n-i_1$. Thus we get the second factor $\eta_{n-i_1}(i_2-i_1,\ldots,i_k-i_1)$.

The result now follows from (\ref{1st-eta-iterate}) by induction on $k$.
%Next, we can use this same argument (only slightly modified) to show that
%\begin{align}
%\eta_{n-i_1}&(i_2-i_1,\ldots,i_k-i_1)\notag \\
%&=A(i_2-i_1)\eta_{(n-i_1)-(i_2-i_1)}    
%                        \left((i_3-i_1)-(i_2-i_1),\ldots,(i_k-i_1)-(i_2-i_1)\right)\notag \\
%&=A(i_2-i_1)\eta_{n-i_2}(i_3-i_2,\ldots,i_k-i_2).\label{2nd-eta-iterate(a)}
%\end{align}
%Combining (\ref{1st-eta-iterate}) and (\ref{2nd-eta-iterate(a)}) delivers
%\begin{equation*}%\label{2nd-eta-iterate(b)}
%\eta_n(i_1,i_2,\ldots,i_k)=A(i_1)A(i_2-i_1)\eta_{n-i_2}(i_3-i_2,\ldots,i_k-i_2).
%\end{equation*}
%To finish the proof, we iterate this procedure %JE Draft 1 1-21-2014 Changed from referencing the line above, since this doesn't really tell us how to iterate I don't think
%a total of $k-2$ more times, and use 
%the fact that $\eta_m(\emptyset)=A(m)$:
%\begin{align*}
%\eta_n(i_1,i_2,\ldots,i_k)=&\, A(i_1)A(i_2-i_1)\eta_{n-i_2}(i_3-i_2,\ldots,i_k-i_2)\\
%&\vdots \\
%=&\, A(i_1)A(i_2-i_1)A(i_3-i_2)\cdots A(i_k-i_{k-1})\eta_{n-i_k}(\emptyset)\\
%=&\, A(i_1)A(i_2-i_1)A(i_3-i_2)\cdots A(i_k-i_{k-1})A(n-i_k).
%\end{align*}
\end{proof}

\begin{corollary}\label{cor-distinguishedrows}
For any row indices $i_1,i_2,\ldots,i_k \in [n-1]$, we have % with $i_1 < i_2 < \cdots < i_k$, we have
\[
\eta_n(i_1,i_2,\ldots,i_k) \leq A(n-k).
\]
\end{corollary}

\begin{proof}
Using Lemmas \ref{lem-increaseAi} and \ref{lem-distinguishedrows}, we have
\begin{eqnarray*}
\eta_n(i_1,i_2,\ldots,i_k) &=& A(i_1)A(i_2-i_1)A(i_3-i_2)\cdots A(i_k-i_{k-1})A(n-i_k)\\
&\leq& A(1)A(1)A(1)\cdots A(1)A(n-k).
\end{eqnarray*}
\end{proof}

\begin{remark}
We should be careful to point out here that Corollary \ref{cor-distinguishedrows} will be of particular utility. It says that the number of monotone triangles containing a prescribed collection of $k$ distinguished rows (\emph{not} including the very last $n$-th row) is at most 
\[
\eta_n(1,2,\ldots,k)=\eta_n(n-k,n-k+1,\ldots,n-1)=A(n-k). 
\]
Said another way, the count of monotone triangles containing $k$ prescribed distinguished rows outside of row $n$ is no more than the count of monotone triangles with at least the top $k$ rows distinguished in addition to row $n$, or similarly the count of monotone triangle with at least the bottom $k+1$ rows distinguished (as both these counts are $A(n-k)$ by Lemma \ref{lem-distinguishedrows}).
\end{remark}

\subsection{The proof}

By considering all $(\tau_1,\ldots,\tau_r)$ such that \emph{exactly} one of them is $\tau_{\min}$, we have
\begin{equation}\label{eqn-lower bound}
p_{\min} \geq \frac{r(A(n)-1)^{r-1}}{A(n)^r} = \frac{r(1-o(1))}{A(n)}.
\end{equation}
We will now produce a matching upper bound.  Throughout the proof, we assume that $n$ is large enough to support our assertions.

Recall that row $i$ in a monotone triangle $\tau$ is distinguished if it consists of entries $1$, $2$, $\ldots$, $i$.  Notice that in order for $\inf(\tau_1,\ldots,\tau_r)=\tau_{\min}$, we must have each row distinguished in at least one of $\tau_1, \ldots, \tau_r$. This observation leads us to focus on the locations of the distinguished rows among the $\tau_j$.

For an $r$-tuple $(\tau_1,\ldots, \tau_r)$ with $\inf(\tau_1,\ldots,\tau_r) = \tau_{\min}$, let $D_1,\ldots,D_r$ be subsets of $[n]$ so that $D_j=\calD(\tau_j)$ is the set of distinguished rows of $\tau_j$. Then the conditions
\begin{equation}\label{infCondition}
D_1 \cup \cdots \cup D_r = [n]
\end{equation}
and 
\begin{equation}\label{infCondition2}
n \in D_1 \cap \cdots \cap D_n
\end{equation}
must both hold. Using a trivial upper bound, there are at most $2^{rn}$ ways of choosing an $r$-tuple $(D_1,\ldots,D_r)$ satisfying (\ref{infCondition}) and (\ref{infCondition2}).

Much of the rest of the proof is devoted to finding an upper bound on the number of $(\tau_1,\ldots, \tau_r)$ with $\inf(\tau_1,\ldots,\tau_r) = \tau_{\min}$ that correspond to a particular $(D_1,\ldots,D_r)$.  We will do this by producing an upper bound that holds for an entire class of $r$-tuples $(D_1,\ldots,D_r)$, and so we define these classes now.

For $i=0,\ldots,6r$, let $\calC_{n-i} = \calC_{n-i}(r)$ be the collection of all $(\tau_1,\ldots,\tau_r)$ with $\inf(\tau_1,\ldots,\tau_r)=\tau_{\min}$ so that some $D_j$ (associated to $\tau_j$) has \emph{exactly} $n-i$ consecutive elements from $[n]$, where $n-i$ is maximal %AH Draft 1 1-20-14
in $D_j$ (and so, in particular, $D_j$ has no larger amount of consecutive elements from $[n]$).  In other words, $\calC_{n-i}$ consists of those $r$-tuples $(\tau_1,\ldots,\tau_r)$ with $\inf(\tau_1,\ldots,\tau_r)=\tau_{\min}$ so that some $\tau_j$ has a block of exactly $n-i$ consecutive distinguished rows. (It should be emphasized that whenever we say ``consecutive'' for the remainder of this paper, it should be interpreted as maximally consecutive as we have done here.) See Figure \ref{fig:ex}. Then let $\calC_{\leq n-6r-1}$ be the collection of all $(\tau_1,\ldots,\tau_r)$ with $\inf(\tau_1,\ldots,\tau_r)=\tau_{\min}$ so that all $D_j$ have at most $n-6r-1$ consecutive elements from $[n]$.  Clearly, for any $(\tau_1,\ldots,\tau_r)$ with $\inf(\tau_1,\ldots,\tau_r)=\tau_{\min}$ we have $(\tau_1,\ldots,\tau_r) \in \calC_{n} \cup \cdots \cup \calC_{n-6r} \cup \calC_{\leq n-6r-1}$. 

\begin{figure}[ht]
\centering
\[
\begin{array}{c}
\tau_1 =
\begin{tabular}{ ccccccc }
  &   &   & 1 &   &   &   \\
  &   & 1 &   & 2 &   &   \\
  & 1 &   & 2 &   & 4 &   \\
1 &   & 2 &   & 3 &   & 4 \\
\end{tabular}
\qquad
\tau_2 =
\begin{tabular}{ ccccccc }
  &   &   & 2 &   &   &   \\
  &   & 1 &   & 3 &   &   \\
  & 1 &   & 2 &   & 3 &   \\
1 &   & 2 &   & 3 &   & 4 \\
\end{tabular}
\end{array}
\]
\caption{An example $(\tau_1,\tau_2)$ with $\inf(\tau_1,\tau_2) = \tau_{\min}$. In this example $D_1 = \{1,2,4\}$, $D_2 = \{3,4\}$, and $(\tau_1,\tau_2) \in \mathcal{C}_{4-2}$.}
\label{fig:ex}
\end{figure}

First, notice that 
\begin{equation}\label{eqn-case1}
|\calC_{n}| \leq rA(n)^{r-1},
\end{equation}
which follows since $(\tau_1,\ldots,\tau_r) \in \calC_n$ implies that $D_j = [n]$ for some $j$.  That is, some $\tau_j = \tau_{\min}$ in this case. Our goal is to show that the number of elements in the remaining classes is small relative to this upper bound.  

Next, consider $\calC_{n-1}$.  How many $(\tau_1,\ldots,\tau_r)$ are in $\calC_{n-1}$?  There are $r$ choices for the monotone triangle $\tau_j$ which will have the $n-1$ consecutive distinguished rows.  In fact, the $n-1$ consecutive distinguished rows must be rows $2,3,\ldots,n$, as $n \in D_j$, and so the first row of $\tau_j$ must be a $2$.  There are $r-1$ choices for $k$ such that $\tau_k$ has a distinguished first row.  Once $k$ is fixed, by Lemma \ref{lem-distinguishedrows} there are $A(n-1)$ possible monotone triangles $\tau_k$ satisfying this condition, and we place no restriction on the remaining $\tau_{\ell}$.  This means that 
\begin{equation}\label{eqn-case2}
|\calC_{n-1}| \leq r(r-1)A(n-1)A(n)^{r-2}.
\end{equation}

We'll consider $\calC_{n-i}$ for $i=2,\ldots,6r$ via the following lemma.

\begin{lemma}\label{lem-small classes}
Let $r \geq 2$ and $i \in \{2,3,\ldots,6r\}$.  Then
\[
|\calC_{n-i}| \leq  r(r-1)^2 i A(i+1)A(i)A(n-i+1)A(n)^{r-2}.
\]
\end{lemma}

\begin{proof}
Choose (in $r$ ways) $j$ so that $\tau_j$ has exactly $n-i$ consecutive distinguished rows.  There are $i$ places for the $n-i$ consecutive distinguished rows to occur (they can't be rows $i,\ldots,n-1$). First assume that the block of $n-i$ rows is not %AH Draft 1 1-20-14
the first or last $n-i$ rows; we will deal with these two cases later.  Since there are $n-i$ distinguished rows in $[n-1]$, by Corollary \ref{cor-distinguishedrows} there are at most $A(i)$ possibilities for a monotone triangle $\tau_j$ that has a fixed placement of the $n-i$ consecutive distinguished rows.  

Next, we choose $k \neq j$ corresponding to a monotone triangle $\tau_k$ that has a distinguished row directly preceding the $n-i$ consecutive distinguished rows of $\tau_j$, and similarly we choose $\ell \neq j$ corresponding to a monotone triangle $\tau_\ell$ (where potentially $\ell = k$) that has a distinguished row directly following the $n-i$ consecutive distinguished rows of $\tau_j$. Let $i_1$ be the number of undetermined rows that precede the $n-i$ distinguished rows of $\tau_j$, and let $i_2 = i-i_1$.  In other words, $\tau_k$ has row $i_1$ distinguished and $\tau_{\ell}$ has row $n-i_2+1$ distinguished.

If $k \neq \ell$, then by Lemma \ref{lem-distinguishedrows} there are $A(i_1)A(n-i_1)A(n-i_2+1)A(i_2-1)$ choices for monotone triangles $\tau_k$ and $\tau_\ell$ subject to these restrictions.  
As $i_1 + i_2 = i$ and $i_1, i_2 \geq 1$, repeated applications of Lemma \ref{lem-increaseAi} delivers
\[
A(i_1)A(i_2-1) \leq A(i-1)
\]
and
\begin{equation*}\label{eqn-ub first 3 rows}
A(i_1)A(n-i_1)A(n-i_2+1)A(i_2-1) \leq A(i-1)A(n-i+1)A(n).  
\end{equation*}

If $k = \ell$, then similarly we have at most $A(i-1)A(n-i+1)$ choices for $\tau_k$.
Finally, we let the remaining $r-3$ ($r-2$, if $k=\ell$) monotone triangles be arbitrary. 

Putting these pieces together, in these $i-2$ cases we obtain the upper bound
\begin{equation}\label{lemCase1}
r(r-1)^2 A(i)A(i-1)A(n-i+1)A(n)^{r-2}.
\end{equation}

Now what about the case where the block of $n-i$ distinguished rows is at the bottom of $\tau_j$? Here, we need only choose 
$k \neq j$ so that $\tau_k$ will have distinguished row $i$. Then, $\tau_j$ can be completed outside of its bottom $n-i$ 
distinguished rows in $A(i+1)$ ways, and $\tau_k$ can be completed outside of its distinguished $i$-th row in $A(i)A(n-i)$ ways. Letting the other 
$n-2$ monotone triangles be %AH Draft 1 1-20-14
arbitrary delivers the bound 
\begin{equation}\label{lemCase2}
r(r-1)A(i+1)A(i)A(n-i)A(n)^{r-2}
\end{equation}
in this case. A similar analysis shows that the case where the block of $n-i$ distinguished rows is at the top of $\tau_j$ delivers the bound
\begin{equation}\label{lemCase3}
r(r-1)A(i)A(i-1)A(n-i+1)A(n)^{r-2}. %AH Draft 1 1-21-14 (switched A(i) and A(i-1) as A(i) comes first from \tau_j)
\end{equation}
Combining the $i-2$ ``intermediate'' cases and their common bound given by (\ref{lemCase1}) with the top/bottom cases and their bounds given by (\ref{lemCase2}) and (\ref{lemCase3}), we obtain
\begin{align*}
|\calC_{n-i}| & \leq r(r-1)^2 (i-2) A(i)A(i-1)A(n-i+1)A(n)^{r-2} \\
&  \hspace{1in} + r(r-1)A(i+1)A(i)A(n-i)A(n)^{r-2}\\
&  \hspace{1in} + r(r-1)A(i)A(i-1)A(n-i+1)A(n)^{r-2}\\ %AH Draft 1 1-21-14
& \le r(r-1)^2 i A(i+1)A(i)A(n-i+1)A(n)^{r-2},
\end{align*}
as claimed.
\end{proof}

\bigskip

Finally, we need to bound $|\calC_{\leq n-6r-1}|$.  Here, we already know that there are at most $2^{rn}$ possible choices for an $r$-tuple $(D_1,\ldots,D_r)$ satisfying (\ref{infCondition}) and (\ref{infCondition2}).  We fix such a $(D_1,\ldots,D_r)$ and bound the number of $r$-tuples $(\tau_1,\ldots,\tau_r) \in \calC_{\leq n-6r-1}$ %AH Draft 1 1-20-14 with $\inf(\tau_1,\ldots,\tau_r) = \tau_{\min}$ 
that correspond to this fixed $r$-tuple $(D_1,\ldots,D_r)$.

A precise count of these $r$-tuples would be difficult to compute, but we only need an upper bound. But surely this count is at most 
\[
\prod_{j=1}^r \eta_n\left(D_j\setminus \{n\}\right);
\]
indeed, the number of $\tau\in\frakM_n$ with \emph{precisely} the distinguished rows $D_j$ is certainly at most the number of $\tau\in\frakM_n$ such that $D_j\subseteq \calD(\tau)$. Then, from Corollary \ref{cor-distinguishedrows} we have 
\begin{equation}\label{eqn-<n-6r 1}
\prod_{j=1}^r \eta_n\left(D_j\setminus \{n\}\right) \leq \prod_{j=1}^r A(n-|D_j|+1).
\end{equation}

Now let index $m$ be such that $|D_m|$ is maximal. First, we'll suppose that $|D_m| \leq n-6r$. Notice that 
\begin{equation}\label{eqn-lb D_i}
(|D_1|-1)+(|D_2|-1) + \cdots + (|D_r|-1) \geq n-1
\end{equation}
which implies that $$|D_m|\ge \frac{n-1}{r}+1=\frac{n+r-1}{r}\ge\frac{n}{r},$$
and so
\[
|D_m|\ge \ceil*{n/r}.
\]
Then by Lemma \ref{lem-increaseAi} we have
\begin{equation}\label{eqn-<n-6r 1a}
\prod_{j=1}^r A(n-|D_j|+1) \leq A(n-|D_m|+1)A(\delta)A(n)^{r-2},
\end{equation}
where $\delta:=\max\{n-|D_1|-|D_2|-\cdots-|D_r|+r+(|D_m|-1),0 \}$. Indeed, let $m_2\neq m$ denote the index of the second largest value among $|D_1|,\ldots,|D_n|$. All we have done in order to obtain (\ref{eqn-<n-6r 1a}) is to leave the $m$-th factor alone while repeatedly increasing the argument sizes of the $r-2$ factors with index not in $\{m,m_2\}$, one at a time, always at the expense of the factor with index $m_2$ (via Lemma \ref{lem-increaseAi}). If it should happen that factor $m_2$'s argument decreases to $0$ in the process, we then leave that argument alone and simply continue to increase the other $r-2$ arguments (other than the $m$-th) all the way up to $n$.

Then using $|D_m|\le n-6r$, (\ref{eqn-lb D_i}), and (\ref{eqn-<n-6r 1a}), we see that
\begin{eqnarray}
\nonumber \prod_{j=1}^r A(n-|D_j|+1)&\leq&A(n-|D_m|+1)A(\delta)A(n)^{r-2}\\
\nonumber &\leq& A(n-|D_m|+1)A(|D_m|)A(n)^{r-2}\\
\label{eqn-<n-6r 2} &\leq& A(6r+1)A(n-6r)A(n)^{r-2},
\end{eqnarray}

\noindent where the last inequality uses Lemma \ref{lem-increaseAi} along with $\ceil*{n/r} \leq |D_m| \leq n-6r$ (and so $A(n-\ceil*{n/r}+1)A(\ceil*{n/r}) \leq A(6r+1)A(n-6r)$ for large enough $n$).

If instead $|D_m| > n-6r$, then we consider the $n-6r$ consecutive rows $3r,3r+1,\ldots,n-3r-1$.  Since $\tau_m$ cannot have $n-6r$ consecutive distinguished rows, by (\ref{infCondition}) we know that some other $\tau_k$ must have a distinguished row from among rows $3r,3r+1,\ldots,n-3r-1$.  Corollary \ref{cor-distinguishedrows} gives at most $A(n-|D_m|+1) \leq A(6r)$ choices for such a $\tau_m$, and Lemmas \ref{lem-distinguishedrows} and \ref{lem-increaseAi} give at most $A(3r)A(n-3r)$ choices for such a $\tau_k$. Putting no restrictions on the remaining monotone triangles,  we have at most
\begin{equation}\label{eqn-<n-6r 3}
A(6r)A(3r)A(n-3r)A(n)^{r-2}
\end{equation}
monotone triangles corresponding to $(D_1,\ldots,D_r)$ where $|D_m| > n-6r$.  Since there are at most $2^{rn}$ $r$-tuples $(D_1,\ldots,D_r)$ and from (\ref{eqn-<n-6r 2}) and (\ref{eqn-<n-6r 3}) each of these corresponds to at most $A(6r)A(3r)A(n-3r)A(n)^{r-2}$ $r$-tuples $(\tau_1,\ldots,\tau_r)$ with $\inf(\tau_1,\ldots,\tau_r) = \tau_{\min}$, for large enough $n$ we have
\begin{equation}\label{eqn-case3}
|\calC_{\leq n-6r-1}| \leq 2^{rn}A(6r)A(3r)A(n-3r)A(n)^{r-2}.
\end{equation}

\medskip

We are now ready to finish the calculation. We have (with the main inequalities justified below, and all implied constants depending on $r$)
\begin{eqnarray}
\nonumber p_{\min} &=& \frac{|\{(\tau_1,\ldots,\tau_r) \colon \inf(\tau_1,\ldots,\tau_r)=\tau_{\min}\}|}{A(n)^r}\\
\nonumber &\leq& \frac{|\calC_n| + |\calC_{n-1}| + \cdots + |\calC_{n-6r}| + |\calC_{\leq n-6r-1}|}{A(n)^r}\\
\label{eqn-1} &\leq& \frac{rA(n)^{r-1} + O( A(n-1)A(n)^{r-2})+O( 2^{rn}A(n-3r)A(n)^{r-2})}{A(n)^{r}}\\
\label{eqn-2}&\leq& \frac{r +  O((2/3)^{n}) + O(2^{rn}(2/3)^{3rn})}{A(n)}\\
\label{eqn-upper bound} &=& \frac{r(1+o(1))}{A(n)},
\end{eqnarray}
where (\ref{eqn-1}) follows from (\ref{eqn-case1}), (\ref{eqn-case2}), Lemma \ref{lem-increaseAi}, Lemma \ref{lem-small classes} and (\ref{eqn-case3}), while (\ref{eqn-2}) follows from Lemma \ref{lem-ratio}. Therefore, from (\ref{eqn-lower bound}) and (\ref{eqn-upper bound}) we have
\[
p_{\min} \sim \frac{r}{A(n)}.
\]

\section{Second-order term}\label{sec-higher order terms}

In this section, we describe how the ideas from the proof of Theorem \ref{thm-main} can be generalized to produce a second-order term for $p_{\min}$.  By inspection, there are $1$, $1$, and $6$ monotone triangles with a maximal block of exactly $n$, $n-1$, and $n-2$ consecutive distinguished rows, respectively.  In particular, 
\begin{equation}\label{eqn-n-3}
\begin{tabular}{c}
\text{there are $A(n)-8$ monotone triangles of size $n$}\\
\text{with at most $n-3$ consecutive distinguished rows.}
\end{tabular}
\end{equation}

Let $\tau_{\min}'$ denote the monotone triangle obtained from $\tau_{\min}$ by changing the top row to a $2$, and $\tau_{\min}''$ denote the monotone triangle %JE Draft 1 1-21-2014 added `denote the...'
obtained from $\tau_{\min}$ by changing row $n-1$ to $1$, $2$, $3$, $\cdots$, $n-3$, $n-2$, $n$.  It is important to note that $\tau_{\min}'$ and $\tau_{\min}''$ each have more than $n-3$ consecutive distinguished rows. 

We produce a lower bound as follows: Consider all $(\tau_1,\ldots,\tau_r)$ such that: 
\begin{enumerate}
	\item one $\tau_j$ is $\tau_{\min}$, and the others have at most $n-3$ consecutive distinguished rows; or
	\item one $\tau_j$ is $\tau_{\min}'$, a second $\tau_k$ has row $1$ distinguished \emph{and} has at most $n-3$ consecutive distinguished rows, and the remaning $\tau_{\ell}$ do not have row $1$ distinguished \emph{and} have at most $n-3$ consecutive distinguished rows; or
	\item one $\tau_j$ is $\tau_{\min}''$ , a second $\tau_k$ has row $n-1$ distinguished \emph{and} has at most $n-3$ consecutive distinguished rows, and the remaining $\tau_{\ell}$ do not have row $n-1$ distinguished \emph{and} have at most $n-3$ consecutive distinguished rows. 
\end{enumerate} 

\noindent Notice that the collections of $r$-tuples in the three cases above are disjoint, since exactly one monotone triangle $\tau_j$ will be either $\tau_{\min}$, $\tau_{\min}'$, or $\tau_{\min}''$. Using (\ref{eqn-n-3}), there are $r(A(n)-8)^{r-1} = rA(n)^{r-1} - \Theta(A(n)^{r-2})$ $r$-tuples $(\tau_1,\ldots,\tau_r)$ in the first collection.  

For the second collection, there are $r$ choices for the index $j$ followed by $r-1$ choices for $k$. To compute the number of possibilities for $\tau_k$ in this collection we use a union bound: subtract the total number of monotone triangles that \emph{do not} have row 1 distinguished (there are $A(n)-A(n-1)$ of these) from the number of monotone triangles with at most $n-3$ consecutive distinguished rows given by (\ref{eqn-n-3}) to obtain at least $(A(n) - 8) - (A(n)-A(n-1)) = A(n-1)-8$ choices for $\tau_k$. %This is not an exact count, but a lower bound will do here. 
Lastly, from (\ref{eqn-n-3}) and another union bound there are at least $(A(n)-8)-A(n-1)$ choices for each of the $r-2$ remaining monotone triangles $\tau_{\ell}$.

The enumeration for the third collection is very similar. There are $r$ choices for $j$, $r-1$ choices for $k$, at least $A(n-1) - 8$ choices for $\tau_k$, and at least $A(n)-A(n-1)-8$ choices for each of the $r-2$ remaining monotone triangles $\tau_{\ell}$. Therefore,
\begin{eqnarray}
\nonumber p_{\min} &\geq& \frac{r(A(n)-8)^{r-1} + r(r-1)(A(n-1)-8)(A(n)-A(n-1)-8)^{r-2}}{A(n)^r}\\
\nonumber && \quad +\,\, \frac{r(r-1)(A(n-1)-8)(A(n)- A(n-1)-8)^{r-2}}{A(n)^r}\\
\label{eqn-2ndLB} &=& \frac{rA(n)^{r-1} + 2r(r-1)A(n-1)A(n)^{r-2} - O(A(n-1)^2 A(n)^{r-3})}{A(n)^r}
\end{eqnarray}

\bigskip

For a corresponding upper bound, we use 
\[
|\calC_{n}| \leq rA(n)^{r-1} \textrm{ and } |\calC_{n-1}| \leq r(r-1)A(n-1)A(n)^{r-2},
\]
as before.  We can refine $\calC_{n-2}$ to find those $r$-tuples with a $\tau_j$ having the top $n-2$ rows distinguished (giving $1$ choice for the last entry in row $n-1$ of $\tau_j$, as row $n-1$ is \emph{not} distinguished; i.e., $\tau_j=\tau_{\min}''$), and a $k \neq j$ so that  $\tau_k$ has row $n-1$ distinguished. There are $r$ choices for $j$, $r-1$ choices for $k$,  $A(n-1)$ choices for $\tau_{k}$, and we let the rest of the choices be arbitrary.  This gives an upper %JE Draft 1 1-21-2014 added `upper'
bound of $r(r-1)A(n-1)A(n)^{r-2}$ for the size of this subset of $\calC_{n-2}$.  From the proof of Lemma \ref{lem-small classes} the remaining subset of $\calC_{n-2}$ has size at most $O(A(n-2)A(n)^{r-2})$ (specifically, see the discussion around (\ref{lemCase2})).

Using the bounds for $|\calC_{n-3}|,\ldots,|\calC_{n-6r}|$ obtained in Lemma \ref{lem-small classes} and the bound for $|\calC_{\leq n-6r-1}|$ from (\ref{eqn-case3}), we have
\begin{eqnarray}
\nonumber p_{\min} &\leq& \frac{rA(n)^{r-1} + 2r(r-1)A(n-1)A(n)^{r-2}}{A(n)^r}\\
\nonumber && \quad +\,\, \frac{O(A(n-2)A(n)^{r-2})+O(2^{rn}A(n-3r)A(n)^{r-2})}{A(n)^{r}}\\
\nonumber &=& \frac{rA(n)^{r-1} + 2r(r-1)A(n-1)A(n)^{r-2}}{A(n)^r}\\
\label{big-O absorb} && \quad +\,\, \frac{O(A(n-2)A(n)^{r-2})+O(2^{rn}A(n-2)(2/3)^{(3r-2)n}A(n)^{r-2})}{A(n)^{r}}\\ %\binom{n-2}{2} - \binom{n-2r+1}{2}}A(n)^{r-2})}{A(n)^{r}}\\
\label{eqn-2ndUB} &=& \frac{rA(n)^{r-1} + 2r(r-1)A(n-1)A(n)^{r-2} + O(A(n-2)A(n)^{r-2})}{A(n)^{r}};
\end{eqnarray}

\noindent here, moving from line (\ref{big-O absorb}) to (\ref{eqn-2ndUB}) we have used $r\ge 2$.

Notice that %AH Draft 1 1-20-14
$A(n-1)^2\le A(n-2)A(n)$ (via Lemma \ref{lem-increaseAi}), and so (\ref{eqn-2ndLB}) and (\ref{eqn-2ndUB}) imply that
\[
p_{\min} = \frac{rA(n)^{r-1} + 2r(r-1)A(n-1)A(n)^{r-2} + \Theta(A(n-2)A(n)^{r-2})}{A(n)^r}.
\]

%=================================================================================
%References
%=================================================================================

\end{document}